\documentclass[11pt,twoside]{article}
\usepackage{amsmath,amssymb,amsthm,color}
\usepackage{srcltx}
\usepackage{stmaryrd}
\usepackage{enumerate}
\usepackage{graphicx}
\usepackage{hyperref}

\newcommand{\dis}{\displaystyle}

\numberwithin{equation}{section}
\newtheorem{theorem}{Theorem}[section]

\newtheorem{lemma}[theorem]{Lemma}
\newtheorem{proposition}[theorem]{Proposition}
\theoremstyle{remark}
\newtheorem{remark}[theorem]{Remark}

\newcommand{\norm}[1]{ \| #1  \|}
\newcommand{\Norm}[1]{\left \| #1 \right \|}
\newcommand{\bka}[1]{{\langle #1 \rangle}}

\newcommand{\abs}[1]{\left | #1 \right |}

\newcommand{\R}{\mathbb{R}}
\newcommand{\C}{\mathbb{C}}

\renewcommand{\O}{\mathcal{O}}

\newcommand{\dist} {\mathop{\mathrm{dist}}}
\newcommand{\Id}{\mathrm{Id}}
\newcommand{\pd}{\partial}
\newcommand{\etat}{\tilde{\eta}}

\renewcommand{\epsilon}{\varepsilon}

\newcommand{\laplacian}{\Delta}

\newcommand{\real}{{\mathbb R}}
\newcommand{\complex}{{\mathbb C}}
\newcommand{\twovec}[2]{\left[\begin{matrix} #1 \\ #2 \end{matrix}\right]}

\newcommand{\svec}[1]{\left \llbracket #1 \right \rrbracket}
\newcommand{\inner}[2]{\left\langle{ #1 },{ #2 }\right\rangle}
\newcommand{\boldL}{{\bf L}}
\newcommand{\boldI}{{\bf I}}
\newcommand{\boldV}{{\bf V}}
\newcommand{\scrptL}{{\mathcal L}}
\newcommand{\J}{{\bf J}}
\newcommand{\K}{{\bf K}}
\newcommand{\EVp}{{e_+}}
\newcommand{\EVm}{{e_-}}
\newcommand{\EVpm}{{e_\pm}}
\newcommand{\Yp}{{Y_+}}
\newcommand{\Ym}{{Y_-}}

\newcommand{\Yr}{{Y_{re}}}
\newcommand{\Yi}{{Y_{im}}}
\newcommand{\bigO}{{\mathcal O}}

\newcommand{\calI}{{\mathcal I}}
\newcommand{\calN}{{\mathcal N}}
\newcommand{\calK}{{\mathcal K}}
\newcommand{\calP}{{\mathcal P}}
\newcommand{\Lloc}{{L^q}}
\newcommand{\TEx}{{T_{exit}}}
\newcommand{\TC}{{T_{crit}}}
\newcommand{\TD}{{T_{dec}}}
\newcommand{\THy}{{T_{hyp}}}
\newcommand{\Bp}{{b_+}}
\newcommand{\Bm}{{b_-}}
\newcommand{\dBp}{{\dot{b}_+}}
\newcommand{\dBm}{{\dot{b}_-}}

\newcommand{\mZero}{m_0}
\newcommand{\deltaone}{{\delta}}
\newcommand{\muBp}{1}
\newcommand{\muBm}{1}
\newcommand{\muP}{{\sigma_p}}
\newcommand{\muQ}{{\sigma_q}}
\newcommand{\muLoc}{\muQ}
\newcommand{\muNon}{\mZero\muLoc}
\newcommand{\muThm}{\mu}
\title{Local Dynamics Near Unstable Branches of NLS Solitons}

\author{V.~Combet\footnote{vianney.combet@math.univ-lille1.fr}\qquad T.-P.~Tsai\footnote{ttsai@math.ubc.ca}\qquad
I.~Zwiers\footnote{zwiers@math.ubc.ca}}
\date{UBC, Vancouver}

\begin{document}

\maketitle

\begin{abstract}
Consider $\left\{\phi_{\omega}\right\}_{\omega \in \calI}$, a branch of unstable solitons of NLS whose linearized operators have one pair of simple real eigenvalues in addition to the zero eigenvalue.
Under radial symmetry and standard assumptions, solutions to initial data from a neighbourhood of the branch either converge to a soliton, or exit a larger neighbourhood of the branch transversally.
The qualitative dynamic near a branch of unstable solitons is irrespective of whether blowup eventually occurs, which has practical implications for the description of blowup of NLS with supercritical nonlinearity.
\end{abstract}

\maketitle

\section{Introduction}

\subsection{Nonlinear Schr\"odinger equation and solitons}

Let us consider the nonlinear Schr\"odinger equation (NLS) with covariant general nonlinearity $g(u):=f(|u|^2)u$,
\begin{equation}\label{Eqn-NLS}
i\partial_t u + \laplacian u - V_0(x) u + f(\abs{u}^2)u =0,\\
\end{equation}
where $u\, :\, (t,x)\in\R\times\R^N\mapsto u(t,x)\in\C$, $N\geq 1$, and
$u(0,\cdot)=u_0\in H^1(\R^N)$. We assume that $f$ is real-valued,
the $C^1$ nonlinearity $g(u)$ is $C^2$ for $u \not =0$, and
$H^1$-subcritical in the sense that:
\begin{align}\label{Eqn-Nonlin}
g(0) = 0, && g'(0) = 0 &&\mbox{ and } && \abs{g''(u)} \lesssim
\abs{u}^{m_1-2} + \abs{u}^{m_2-2} \quad \mbox{for}\ u\neq 0,
\end{align}
with  $1<m_1\leq m_2<m_{\max}$;
$m_{\max}=\frac{N+2}{N-2}$ if $N\geq 3$, and $m_{\max}=+\infty$ if $N=1,2$.
$V_0$~is either zero, or a smooth localized potential such that a range of assumptions hold; see Section \ref{SubSec-Assumptions}.
These equations are Hamiltonian, with the following conserved quantities:
\[\begin{aligned}
&E(u_0) &&=&& E(u)&&:=&& \frac{1}{2}\int \left(\abs{\nabla u}^2 + V_0\abs{u}^2\right) -\int G(u) && \mbox{(energy)}, \\
&M(u_0) &&=&& M(u)&&:=&& \frac{1}{2}\int \abs{u}^2 && \mbox{(mass)},
\end{aligned}\]
where $G(s)=\int_0^s g(\bar{s})d\bar{s}$.
For data in the \emph{energy space} $H^1(\R^N)$, equation \eqref{Eqn-NLS} is locally well posed with a blowup alternative: there exists $T_{\max}>0$ such that
\[
u(t) \in C\left([0,T_{\max}),H^1\right),
\]
and either $T_{\max} = \infty$ or $\lim_{t\to T_{\max}}\norm{u(t)}_{H^1} = \infty$.
The modern proof is due to Kato~\cite{Kato1987}.


Another well known property of \eqref{Eqn-NLS} is the existence of solitary wave (or \emph{soliton}) solutions of the form  $u(t,x)=\phi_{\omega}(x)e^{i\omega t}$. Berestycki and Lions~\cite{BerLions1983} proved that for $V_0 = 0$ there exists a solution $\phi_{\omega}>0$ of
\begin{equation}\label{Eqn-Solitons}
\laplacian\phi_\omega - \omega\phi_\omega +f(\abs{\phi_\omega}^2)\phi_\omega =0,
\end{equation}
provided $\omega>0$ and there exists $u_1>0$ such that
$G(u_1)>\frac{\omega}{2}u_1^2$. Note that this further assumption on
the nonlinearity $g$ is independent of $N$. The positive solution of
\eqref{Eqn-Solitons} is unique and radially symmetric up to translations. See Cazenave's book \cite{Cazenave} and the references therein, particularly McLeod~\cite{McLeod1993}. For $V_0\neq 0$, see Rose and Weinstein~\cite{RosenWeinstein} and articles that cite it.


We are interested in the dynamics and stability of these objects. Because of rotation invariance of equation~\eqref{Eqn-NLS}, the proper notion is \emph{orbital stability}: we say that $\phi_{\omega_0}$ is orbitally stable if for all $\epsilon>0$, there exists $\delta>0$ such that
\[
\norm{u_0-\phi_{\omega_0}}_{H^1}\leq \delta \Longrightarrow \sup_{t>0}\left( \inf_{s\in\R} \Norm{ u(t)-\phi_{\omega_0}e^{is}}_{H^1} \right) \leq\epsilon.
\]
Note that, when the soliton family has the spatial translation symmetry, the solution may move but remain close to the soliton family. In that case, the above concept of orbital stability is not adequate, and the nullspace of the linearized operator (to be discussed in Section~\ref{SubSec-LinOp}) will have extra elements ($\nabla \phi_{\omega}$). One possible remedy is to enlarge the ``orbit'' to contain all translated solitons. Another way, which we choose in this paper, is to restrict ourselves to the cases with no translation. For example, we may assume the perturbation is either radial or even. Alternatively, we may also assume the potential $V_0$ is nonradial and hence the soliton family has no translation symmetry.

We now assume, as is usually the case, that we have a whole \emph{branch} of solitons, that is a family $(\phi_{\omega})_{\omega\in\calI}$ of solitary waves, with $\phi_{\omega}>0$ radial, $\calI=(\underline{\omega},\overline{\omega})$, and a $C^2$ map $\omega\in\calI\mapsto \phi_{\omega}\in H^1(\R^N)$.  Under these assumptions and some generic spectral conditions, for $\omega_0\in\calI$, we have:
\begin{enumerate}[(i)]
\item $\phi_{\omega_0}$ is orbitally stable if ${\left. \frac{d}{d\omega}\int |\phi_{\omega}|^2\, \right|}_{\omega=\omega_0} >0$;
\item $\phi_{\omega_0}$ is orbitally unstable if ${\left. \frac{d}{d\omega}\int |\phi_{\omega}|^2\, \right|}_{\omega=\omega_0} < 0$.
\end{enumerate}

The part (i) is due to Cazenave-Lions~\cite{CazenaveLions1982} and Weinstein~\cite{Weinstein1986}, by variational and energy methods. The part (ii) is due to Shatah and Strauss~\cite{ShatahStrauss}. See also \cite{GSS1987}.
We are interested in \emph{unstable} branches of solitons.

In the case of a pure power \emph{focusing} nonlinearity, $f(s)=+s^{(m-1)/2}$ and $V_0=0$, there exists a profile $\phi_\omega$ for all $\omega>0$.
Moreover, there is scaling invariance, $\phi_{\omega}(x)=\omega^{\frac{1}{m-1}}\phi_1(\sqrt{\omega}x)$, by which we can compute $\norm{\phi_{\omega}}_{L^2} = \omega^{\frac{1}{m-1}-\frac{N}{4}}\norm{\phi_1}_{L^2}$.
We let $m_c=\frac{4}{N}+1$ and note that $\phi_{\omega}$ is stable for all $\omega>0$ when $1<m<m_c$ ($L^2$ subcritical case), and $\phi_{\omega}$ is unstable for all $\omega>0$ when $m_c<m<m_{\max}$ ($L^2$ supercritical case).

When the solitary wave
$\phi_{\omega_0}$ is stable, it is likely that a nearby solution will
relax to some $\phi_{\omega_+}$ as time goes to infinity. The
frequency $\omega_+$ is close to $\omega_0$ but most likely
different. In addition, the convergence happens only locally since
there is radiation going to infinity with unvanishing $L^2$-mass.
Thus one also considers a local concept of stability of the branch: we
say that the branch $(\phi_{\omega})_{\omega\in \calI}$ is {\it
  asymptotically stable} if, for any $u_0$ in a suitable neighbourhood,
\[
\norm{u(t) - \phi_{\omega(t)} e^{i \theta(t)}}_{L^2_{loc}} \longrightarrow 0,
\quad \text{as} \quad t \to \infty,
\]
for some continuous functions $\omega(t) \in \calI$ and $\theta(t) \in \R$.
Asymptotic stability is the main
stability  concept used in the statement of our main theorem.

This paper is organized as follows. Section
\ref{SubSec-Assumptions} will itemize our assumptions in terms of
the nonlinearity and in terms of the evolution operator $\J\boldL$
linearized about $\phi_\omega$.  We then state our main result.
Section \ref{SubSec-Context} will discuss the context for our
theorem and relevant literature. The linearized operator
$\J\boldL$ will be properly introduced in Section
\ref{SubSec-LinOp}, followed by the proof of our main result.

\subsection{Assumptions \& Main Result}
\label{SubSec-Assumptions}

\begin{enumerate}
    \item Solitons exist for $\omega \in \calI$, and, moreover, the
            profiles $\phi_\omega$ and their derivatives $\partial_\omega^a\partial_x^b\phi_\omega$ have exponential decay. We assume the map $\omega\mapsto\phi_\omega\in H^1$ is $C^2$.

    \item For all $\omega\in\calI$, the linearized operator $\J\boldL$ has eigenvalue zero with multiplicity two.  The discrete spectrum contains exactly
            two simple real eigenvalues, $\EVp > 0 > \EVm$, with corresponding eigenfunctions $\Yp$ and $\Ym$ with exponential decay.
            The purely imaginary continuous spectrum is bounded away from zero.
            There is no embedded eigenvalue or resonance in the continuous spectrum or its endpoints.

         As unstable solitons are typically characterized by $\inner{\partial_\omega \phi_\omega}{\phi_\omega}<0$, we assume that $\inner{\partial_\omega \phi_\omega}{\phi_\omega}\neq 0$ with a uniform bound on $\calI$.

    \item Nonlinearity $ g(u):=f(\abs{u}^2)u$ is sufficiently strong.  In addition to \eqref{Eqn-Nonlin}, we require
          \begin{align}
            & m_1 > 1 + \frac{2}{N}\left(1+\sigma_p\right), \label{Eqn-NonlinCondition1} \\
            & m_1 > 1 + \frac{2}{N}\left(1+\frac{2}{\mZero+1}\right), \label{Eqn-NonlinCondition2}\\
            &\mbox{where } \mZero \equiv \min\left\{2, m_1\right\}. \nonumber
          \end{align}
          Here, we use $p$ to denote $\dis  m_2+1 \equiv p < p_{\max} \equiv m_{\max}+1$, and the standard notation $\sigma_r = N\left(\frac{1}{2}-\frac{1}{r}\right)$ and $\frac{1}{r}+\frac{1}{r'}=1$ for any exponent $r$. See equation \eqref{Eqn-SigmaQ}, below, for a related fixed constant $q\in[p,p_{\max})$ determined by both $m_1$ and $m_2$. The choice of $q$ will require $N\geq 2$. 

          Condition \eqref{Eqn-NonlinCondition1} implies $m_1 > \frac{1}{2}+\frac{1}{N}+\frac{\sqrt{N^2+12N+4}}{2N}$, the Strauss exponent. 
          However, we do not approach this value, as condition \eqref{Eqn-NonlinCondition2} implies either $m_1 > 1+\frac{2}{N}\left(\frac{5}{3}\right)$ or $m_1 > \frac{1}{N} + \frac{\sqrt{N^2+6N+1}}{N}$, which are more restrictive. 
          See figures \ref{FigureN3} and \ref{FigureN45}.  We note that conditions \eqref{Eqn-NonlinCondition1} and \eqref{Eqn-NonlinCondition2} admit the entire range of $L^2$-critical and supercritical exponents, $1+\frac{4}{N}\leq m_1 \leq m_2<\frac{N+2}{N-2}$.

          Relations \eqref{Eqn-NonlinCondition1} and \eqref{Eqn-NonlinCondition2} are required to complete bootstrap estimates of the $L^p$ and $\Lloc$ norms, respectively. The arithmetic consequences of \eqref{Eqn-NonlinCondition1} and \eqref{Eqn-NonlinCondition2} are discussed in Section \ref{Section-Numerology}.

\begin{figure*}[htpb]
\centering
\includegraphics[width=0.9\textwidth]{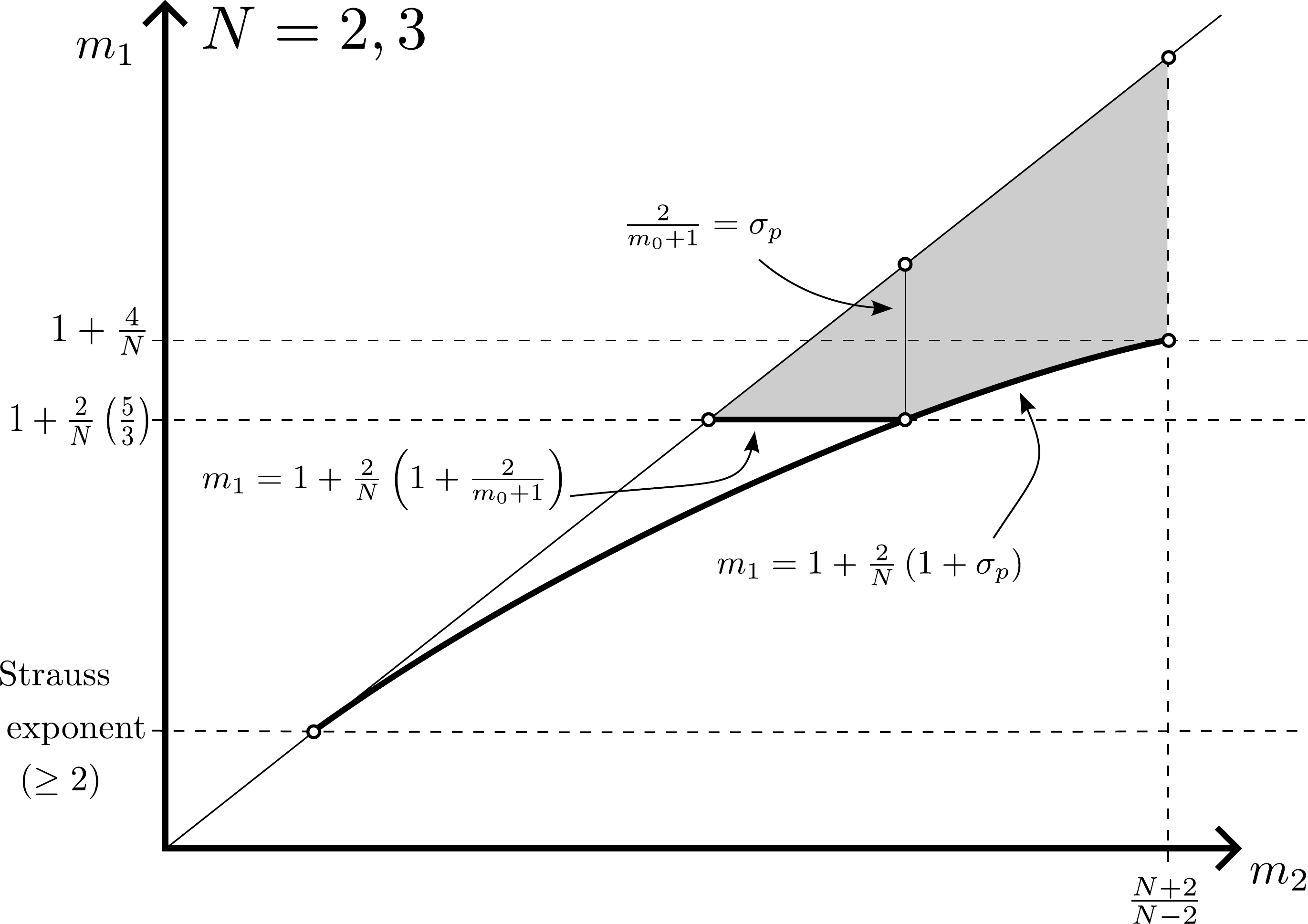} 
\caption{Shaded region illustrates conditions \eqref{Eqn-NonlinCondition1} and \eqref{Eqn-NonlinCondition2} for $N=2,3$.}
\label{FigureN3}
\end{figure*}
\begin{figure*}[htpb]
\centering
\includegraphics[width=0.90\textwidth]{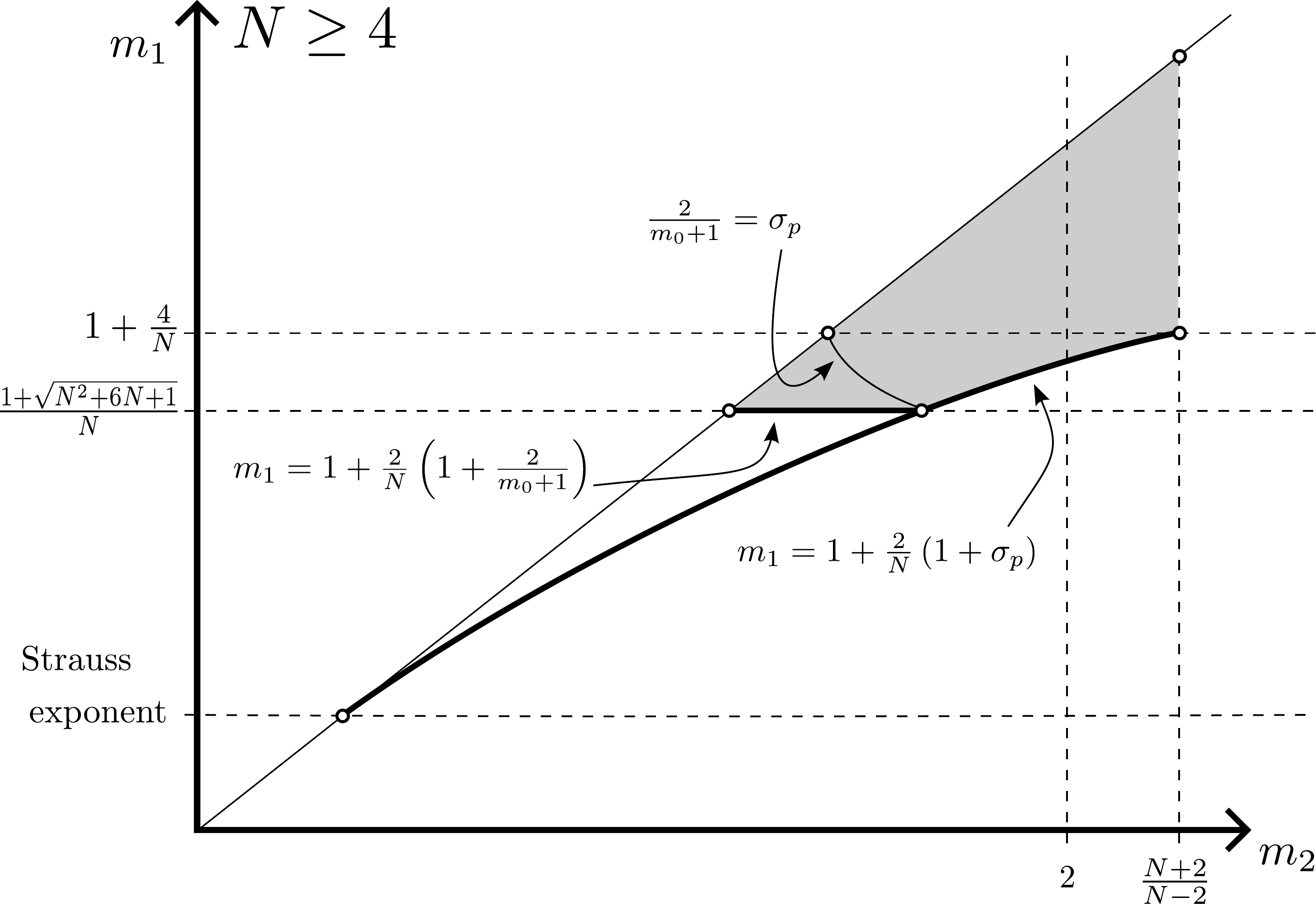}
\caption{Shaded region illustrates conditions \eqref{Eqn-NonlinCondition1} and \eqref{Eqn-NonlinCondition2} for $N\geq 4$.}
\label{FigureN45}
\end{figure*}

    \item There is an adequate dispersive estimate, uniformly in $\omega$:
        \begin{equation}\label{Eqn-AssumedDecay}
            \sup_{\omega\in\calI}\norm{e^{t\J\boldL}\calP_c\xi}_{L^r} \lesssim t^{-\sigma_r}\norm{\xi}_{L^{r'}},
        \end{equation} 
                where $\calP_c$ is the continuous spectral projection with respect to the linearized operator $\J\boldL$ (depends on $\omega$). See \eqref{Eqn-EtaProjection} below for details.
                We only require estimate \eqref{Eqn-AssumedDecay} for $r = p,q$, which excludes endpoint cases.

                For $V_0=0$, \eqref{Eqn-AssumedDecay} is a consequence of the spectral assumption.
                For $N>2$, see Cuccagna~\cite[Corollary 2.2]{Cuccagna2001}, built on the earlier works of Yajima~\cite{Yajima1995no1, Yajima1995no3}.
                For $N=2$, see Cuccagna and Tarulli~\cite{CuccagnaTarulli2009}. 

                The class of non-trivial potentials $V_0$ for which \eqref{Eqn-AssumedDecay} holds is unknown.


\end{enumerate}

All assumptions are known to be true for the monic cubic focusing equation in three dimensions. The lack of embedded eigenvalues is due to a numerically assisted proof by Marzuola and Simpson~\cite{MarzuolaSimpson}.  All assumptions are expected to be true for any monic $L^2$-supercritical and energy-subcritical equation, and for their perturbations.  Therefore, after rescaling, the assumptions should hold for sums of two monic nonlinearities in certain ranges of $\omega$.  Indeed, the spectral assumption is partially known for the cubic-quintic nonlinearity in exactly such a situation. See Asad and Simpson~\cite{AsadSimpson}. See Hundertmark and Lee~\cite{HundertmarkLee2007} regarding the decay of $\Yp$, $\Ym$.

{{
\begin{theorem}\label{Thm-MainResult}
Let $1<N<\infty$ and suppose the above assumptions are valid for
$\omega$ in an interval $\calI$ with uniform estimates. There
exist positive constants $\overline{\alpha_0} =
\overline{\alpha_0}(f,\calI)$ and $C=C(f,\calI)$ such that the
following holds.

For any $\alpha_0 \in (0,\overline{\alpha_0}]$, consider initial
data $u_0$ for which there exist $\theta_0\in\R$ and $\omega_0\in
\calI$ with dist$(\omega_0,\partial \calI)>C \alpha_0$, such
that:
\[0<\alpha \equiv \norm{u_0-\phi_{\omega_0}e^{i\theta_0}}_{H^1\cap L^1}
                        < \alpha_0.\]
Either:
\begin{enumerate}
    \item {\bf (escape case)}\quad there exist $\omega_+\in\calI$,  $\abs{\omega_0-\omega_+}\le C \alpha_0$, $\theta_+\in [0,2\pi]$ and finite time $\TEx > 0$, such that,
    \[ \norm{u(\TEx)-\phi_{\omega_+}e^{i\theta_+}}_{L^2_{loc}} = \inf_{\omega\in \calI,\theta\in\R} \norm{u(\TEx)-\phi_{\omega}e^{i\theta}}_{L^2_{loc}} \geq 2\alpha_0,\quad\mbox{ or,}\]
    \item {\bf (convergence case)}\quad there exist $\omega_+\in\calI$,  $\abs{\omega_0-\omega_+}\le C \alpha$, $t\mapsto \theta(t)$ continuous and $\muThm=\muThm(f)>0$, such that,
    \[\norm{u(t)-\phi_{\omega_+}e^{i\theta(t)}}_{L^p} \leq C\alpha \bka{t}^{-\muThm} \quad\mbox{ for all }t>0.\]
\end{enumerate}
\end{theorem}
}}
Recall the standard notation $\bka{t}=\sqrt{1+t^2}$ for $t\in\R$, and that $p=m_2+1$ is the largest exponent in the potential energy. Since $L^p\hookrightarrow L^2_{loc}$, the convergence case is an asymptotic stability result. Let us emphasize that the estimates must be uniform over $\calI$.  In particular, the real eigenvalues $e_+,e_-$ and $\inner{\partial_\omega \phi_\omega}{\phi_\omega}$ are uniformly bounded away from zero.
The restriction dist$(\omega_0,\partial \calI)>C \alpha_0$
ensures that both $\omega_+\in\calI$ and $\omega(t)\in\calI$,
where $\omega(t)$ will be defined by \eqref{Eqn-Decomp1}.

\subsection{Context and Importance}\label{SubSec-Context}

Asymptotic stability of orbitally stable solitary waves is well
studied and has a vast, growing literature, initiated by Soffer-Weinstein~\cite{SofferWeinsteinI,SofferWeinsteinII} and Buslaev-Perelman~\cite{Buslaev1995}.

%
%

For unstable solitary waves, the classical result of
Glassey~\cite{Glassey} shows the existence of finite time blow-up for
pure power nonlinearities, with no description on the nature of the
blow-up. The general result of Shatah-Strauss~\cite{ShatahStrauss}
exhibits solutions which are initially arbitrarily close to the
solitary waves but leave their neighbourhood in finite times.  Also
see Comech-Pelinovsky~\cite{ComechPelinovsky2003}, who give a similar
result when $\omega_0$ is the borderline between stable and unstable
branches.  There are also results showing the existence of stable
(or center-stable) manifolds, solutions which converge to the unstable solitary wave, see
e.g.~\cite{TsaiYau2,KriegerSchlag2006,Schlag2009,Beceanu2012}.
Solutions on a stable manifold are necessarily nongeneric. Indeed,
there are few results addressing {\it all} solutions with initial data
in a neighbourhood of unstable solitary waves.  There are some
exceptions:
\begin{enumerate}
\item Small solitary waves obtained from a linear potential.

A complete description of asymptotic behaviour is known in some cases.  See \cite{TsaiYau4, NakanishiPhanTsai}.

\item Solitary waves of the pure-power $L^2$-critical nonlinearity, $g(u) = |u|^\frac{4}{N}u$.

  There exist disjoint open subsets $\calK^\pm\subset H^1$ such that the solitary waves belong to $\overline{\calK^+}\cap\overline{\calK^-}$.  In some cases, we know that solutions in $\calK^+$ scatter. See Killip, Tao and Visan~\cite{KillipTaoVisan} and papers that refer to it. On the other hand, solutions in $\calK^-$ blowup in finite time, as proved by Merle and Rapha\"el~\cite{Merle2005a} following a couple decades of careful asymptotic arguments. These blowup solutions are precisely described in terms of a soliton profile and a tracking error.  The tracking error is arbitrarily small, and converges, in $L^2$.  Indeed, it converges in $H^1$ outside any ball of fixed radius around the blowup point. The primary growth of $H^1$ norm is captured by the soliton profile $\phi_{\omega(t)}$, for which $\norm{\phi_{\omega(t)}}_{\dot{H}^1}\to\infty$ as $\omega(t)\to\infty$.

  The $L^2$-critical nonlinearity is a degenerate case with physical relevance.  The related literature is very large.

\item Solitary waves of the $\dot{H}^\frac{1}{2}$-critical
  nonlinearity, $g(u) = |u|^\frac{4}{N-1}u$.



In this case, the product $M(u)E(u)$ is invariant under the natural
scaling.  Duyckaerts, Holmer and Roudenko
\cite{HR2008,DHR2008,DR2012} show that all solutions scatter when
$M(u)E(u)$ is less than that of the ground state, expanding on the
energy-critical argument of Kenig and Merle~\cite{KenigMerle}.
For $N=3$, $g(u) = \abs{u}^2u$, and radial data with $M(u)E(u)$ at
most slightly above that of the ground state, Nakanishi and Schlag
\cite{NakanishiSchlag} show that the sets of data leading to
scattering and blow-up are bordered by the center-stable manifold,
that these three are the only possible positive time asymptotics, and
that all nine possibilities as $t \to \pm \infty$ exist.

\item
For slightly $L^2$-supercritical nonlinearities, $g(u) =
\abs{u}^{\frac{4}{N}+\delta}u$, Merle, Rapha\"el and
Szeftel~\cite{MRS2009} have shown that the $L^2$-critical blowup
regime survives as sets of initial data~${\mathcal O}$, open in $H^1$,
for which blowup occurs and can be described in terms of a member of
the soliton family and a tracking error.
%
%
At blowup time, the tracking error converges in all
subcritical\footnote{For a pure-power nonlinearity, $g(u) =
  \abs{u}^{\alpha}u$, the critical norm is
  $L^{N\alpha/2}$.}  norms to a fixed residue which
is outside the critical space. This agrees with the more general
result of Merle and Rapha\"el~\cite{MR2007_CritNorm} that the critical
norm of radially symmetric blowup solutions is unbounded.  Moreover,
there is a universal lower bound for the size of the residue in
$L^2_{loc}$.
Should blowup with a soliton profile occur in any other
$L^2$-supercritical problems, Theorem \ref{Thm-MainResult} suggests
there should be a similar $L^2_{loc}$ lower bound on any residue.

\end{enumerate}

Should equation \eqref{Eqn-NLS} lead to blowup, it may be structurally perturbed by a vanishing multiple of $u\abs{u}^{m_2-1+}$ to be globally wellposed. As a result, Theorem \ref{Thm-MainResult} shows that the qualitative dynamic near a branch of unstable solitons is universal, irrespective of whether blowup eventually occurs.  We do not exclude the possibility of solutions that blowup with a soliton profile, following an unstable branch of solitons at some distance. Should such solutions exist, Theorem \ref{Thm-MainResult} suggests that the blowup dynamic is only observed once they lie outside a particular neighbourhood of the manifold.

While we are concerned with large solitons, our approach will be similar to the small-soliton case. Kirr, M\i zrak and Zarnescu~\cite{KirrZarnescu2D,KirrMizrak3D,KirrMizrak4D} consider the nonlinear Schr\"odinger equation with potential in dimensions two to five and detail the convergence to a center manifold of small stable solitons.  Since they do not require \eqref{Eqn-NonlinCondition2}, their work admits a larger range of nonlinearities down to the Strauss exponent.  The technique in all three papers is focused on time-dependent linear operators, which we avoid, and is strictly limited to small solitons.  Recent work of Beceanu~\cite{Beceanu} may offer a new route to soliton stability results in the $H^1$-setting. We do not know if Theorem \ref{Thm-MainResult} is optimal.

Our approach recovers the asymptotic stability of large solitary waves with no non-zero eigenvalue, for radial perturbations:
\begin{proposition}\label{Prop-ZeroEVOnly}
Let $1<N<\infty$, suppose that for $\omega\in\calI$ the linearized operator $\J\boldL$ has eigenvalue zero with multiplicity two, no other discrete spectra, and that all other assumptions are valid with uniform estimates.  Then the result of Theorem \ref{Thm-MainResult} holds. Moreover, only the convergence case occurs.
\end{proposition}
For small solitons, potentials $V_0$ with linearized operators as in Proposition \ref{Prop-ZeroEVOnly} exist.  Proposition \ref{Prop-ZeroEVOnly} then recovers a range of nonlinearities covered by Kirr, M{\i}zrak and Zarnescu.

We finally note that, when a certain normal form of a
spectrally-stable sign-changing solitary wave has mixed signs, it is
asserted in Cuccagna \cite[Remark 10.8]{Cuccagna2011} that it is
orbitally unstable and a dichotomy result similar to that of Theorem
\ref{Thm-MainResult} holds.

\section{Decomposition and Algebraic Relations}\label{Section-ProofPrep}
First, in Section \ref{SubSec-LinOp}, we properly introduce the linearized operator $\J\boldL$.  Second, in Section~\ref{SubSec-Decomp}, we decompose the solutions $u(t)$.  This will allow us to phrase the bootstrap argument for Theorem \ref{Thm-MainResult}, explained in Section \ref{SubSec-Bootstrap}. We then introduce particular tools in preparation for the following chapters.  We state the dynamic equations of the modulation parameters in Section \ref{Subsection-DynamicEqns}, the tracking-error equation in Section \ref{Subsection-TrackingError}, and a Buslaev-Perelman decomposition and estimate of the continuous spectral projection operator of $\J\boldL$ in Section \ref{Subsec-BPEstimate}.  Finally, in Section \ref{Section-Numerology}, we introduce decay-rate constants and verify associated arithmetic.

\subsection{Linearized Operator}\label{SubSec-LinOp}

Expand the potential around $\phi$,
\begin{equation}\label{Eqn-Expansion}\begin{aligned}
g(\phi+\epsilon) &= g(\phi) + Dg(\phi)\epsilon + \calN(\epsilon) \\
                 &= \phi f(\abs{\phi}^2)
                    +\left(\epsilon f(\abs{\phi}^2)  + \phi f'(\abs{\phi}^2)(\phi\overline{\epsilon}+\overline{\phi}\epsilon)\right)
                    +\calN(\epsilon).
\end{aligned}\end{equation}
Immediately, we recognize the linearized operator
\[
\scrptL(\epsilon) = -\laplacian \epsilon + V_0\epsilon + \omega\epsilon - \left(\epsilon f(\abs{\phi}^2)  + \phi f'(\abs{\phi}^2)(\phi\overline{\epsilon}+\overline{\phi}\epsilon)\right)
\]
around $\phi$, which satisfies $i\partial_t \epsilon = \scrptL(\epsilon)+\bigO(\epsilon^2)$.
Equation \eqref{Eqn-Nonlin} allows us to estimate the nonlinear
term (e.g. \cite[equation (22)]{KirrMizrak4D}),
\begin{equation}\label{Eqn-NonlinConseq}\begin{aligned}
    \abs{\calN(\epsilon)}  &\lesssim A_1\abs{\phi}^{m_1-2}\abs{\epsilon}^2 + A_2\abs{\phi}^{m_2-2}\abs{\epsilon}^2 + \abs{\epsilon}^{m_1} + \abs{\epsilon}^{m_2},
  \end{aligned}\end{equation}
where $A_j = 0$ if $m_j \le 2$.

\begin{remark}[Complex-valued Functions as Vectors]
Consider complex-valued functions $u = u_1 + i u_2$ and $v=v_1+i v_2$, which we write as $\real^2$-valued functions: $\svec{u} = \twovec{u_1}{u_2}$ and $\svec{v} = \twovec{v_1}{v_2}$. The correct inner product is
\[
\inner{\twovec{u_1}{u_2}}{\twovec{v_1}{v_2}}
    = \inner{u_1}{v_1} + \inner{u_2}{v_2}
    =\int{(u_1\overline{v_1} + u_2\overline{v_2})\,dx}.
\]
We include the complex conjugate of $v$ since we will later consider $\complex^2$-valued functions.
We denote the symplectic operator by
\[\begin{aligned}
\J = \left[\begin{matrix}0&+1\\-1&0\end{matrix}\right].
\end{aligned}\]
\end{remark}

Write $\epsilon = \epsilon_1 + i\epsilon_2$, and represent $\epsilon$ as $\svec{\epsilon} = \twovec{\epsilon_1}{\epsilon_2}$. Then,
\[
\svec{\scrptL(\epsilon)} = \left[\begin{matrix}L_+&0\\0&L_-\end{matrix}\right]\twovec{\epsilon_1}{\epsilon_2}
\equiv \boldL\svec{\epsilon}
\equiv \left((-\laplacian+\omega)\boldI + \boldV\right)\svec{\epsilon},
\]
where
\[\begin{aligned}
L_+ = -\laplacian +V_0 + \omega - \left(f(\phi^2) + 2\phi^2f'(\phi^2)\right),
&&
L_- = -\laplacian +V_0 + \omega - f(\phi^2).
\end{aligned}\]
The term $\boldI$ denotes the identity, and the potential term $\boldV$ has localized support and is written separately in anticipation of equation \eqref{Eqn-EtaTilde}.
Note that $\boldL$ is self-adjoint, and that the linearization of \eqref{Eqn-NLS} near $\phi$ will feature $\J\boldL$. The kernel of $\J\boldL$ can be found by inspection: $L_-\phi = 0$ and $L_+\partial_\omega\phi = -\phi$. In vector notation:
\[\begin{aligned}
\J\boldL\twovec{\partial_\omega\phi}{0} = \twovec{0}{\phi}, && \J\boldL\twovec{0}{\phi} = 0.
\end{aligned}\]
Let $\Yp = \Yr + i\Yi$ be the eigenfunction associated with $\EVp$:
\begin{equation}\label{Eqn-YpSpecifics}\begin{aligned}
\J\boldL\svec{\Yp} = \EVp\svec{\Yp}
&&\Rightarrow&&
L_-\Yi = \EVp\Yr, &&   L_+\Yr = -\EVp \Yi.
\end{aligned}\end{equation}
Let $\Ym$ be an eigenfunction associated with $\EVm$.
From \eqref{Eqn-YpSpecifics}, one can verify $\Ym = \overline{\Yp}$ and $\EVm = -\EVp$.
Since $L_-$ is non-negative, and $L_-\Yi = \EVp \Yr$, we note that $\inner{\Yr}{\Yi} = \frac{1}{\EVp}\inner{L_-\Yi}{\Yi} > 0$. Without loss of generality we specify $2\inner{\Yr}{\Yi} = +1$.
This describes the discrete eigenfunctions of $\J\boldL$. Note that
\begin{equation}\label{Orthog0}\begin{aligned}
\inner{\svec{Y_\pm}}{\J\twovec{0}{\phi}} = 0
&& \text{ and } &&
\inner{\svec{Y_\pm}}{\J\twovec{\partial_\omega\phi}{0}} = 0.
\end{aligned}\end{equation}

\subsection{Orthogonality Conditions}\label{SubSec-Decomp}
We assume the decomposition
\begin{equation}\label{Eqn-Decomp}
u(t) = \left(\phi_{\omega(t)} + a(t)\partial_\omega\phi_{\omega(t)} + \Bp(t)\Yp + \Bm(t)\Ym + \eta(t)\right)e^{i\theta(t)},
\end{equation}
where the modulation parameters $\omega(t), \theta(t), a(t), \Bp(t), \Bm(t) \in \real$ are continuous functions determined by enforcing orthogonality conditions\footnotemark.
To determine the parameters:

\footnotetext{By way of comparison with Buslaev-Sulem, note that the coefficients $b_\pm(t)$ and the components of $\svec{Y_\pm}$ are real-valued. In the case of stability and eigenvalues $\pm i\lambda$, the second component of the eigenfunction in vector form is purely imaginary (corresponding to an entirely real-valued eigenfunction), and the appropriate decomposition is $z(t)\svec{\psi} + \overline{z(t)}\svec{\overline{\psi}}$.}

\begin{equation}\label{Orthog1}
\inner{\svec{\eta}}{\J\twovec{0}{\phi}} = 0
\end{equation}
\begin{equation}\label{Orthog2}
\inner{\svec{\eta}}{\J\twovec{\partial_\omega\phi}{0}} = 0
\end{equation}
\begin{equation}\label{Orthog3}
\inner{\svec{\eta}}{\J\svec{\Yp}} = 0
\end{equation}
\begin{equation}\label{Orthog4}
\inner{\svec{\eta}}{\J\svec{\Ym}} = 0
\end{equation}
Parameters $\omega(t)$ and $a(t)$ are not independent; see the proof of Lemma \ref{Lemma-Modulation}, below. We choose to fix $a(t)=0$. As a consequence, the linearized operator $\J\boldL$, the eigenfunctions $\Yp$, $\Ym$, and their associated eigenvalues, are all themselves functions of time through $\omega(t)$.  To simplify notation, this dependence is usually omitted, as in \eqref{Eqn-Decomp}.
When we consider a fixed operator, associated with some fixed value $\omega(T)$, we will refer to
the associated linearized operator as $\boldL_T$ and the eigenvalues as $e_{\pm,\omega(T)}$.

The chosen orthogonality conditions are with the eigenfunctions of the adjoint of $\J\boldL$. These conditions will allow an easy derivation of the dynamical equations in Section \ref{Subsection-DynamicEqns}.  Indeed, $\eta$ is the projection onto the continuous spectrum of $\J\boldL$,
\begin{equation}\label{Eqn-EtaProjection}
\eta = \calP_c\left(e^{-i\theta(t)}u(t)\right),
\end{equation}
whenever the orthogonality conditions uniquely determine $\eta$.

To see that this is the case, let $X_0$ denote the $\J\boldL$-invariant subspace associated with the generalized kernel, and $X_1$, $X_c$ the subspaces associated with eigenvalues $\EVpm$ and the continuous spectrum respectively, so that $L^2 = X_0 \oplus X_1 \oplus X_c$.
Let $\calP_0$, $\calP_1$ and $\calP_c = \Id -\calP_0-\calP_1$ denote the projection operators onto $X_0$, $X_1$ and $X_c$ respectively. Explicitly,
\begin{equation}\label{Eqn-DiscreteProj}\begin{aligned}
&\calP_0\svec{f} = \frac{1}{\frac{1}{2}\partial_\omega\norm{\phi}_{L^2}^2}\left(
    \inner{\svec{f}}{\J\twovec{\partial_\omega\phi}{0}}\twovec{0}{\phi}
    - \inner{\svec{f}}{\J\twovec{0}{\phi}}\twovec{\partial_\omega\phi}{0}
\right),\\
&\calP_1\svec{f} = \frac{\inner{\svec{f}}{\svec{\Yp}}}{2\inner{\Yr}{\Yi}}\J\svec{\Ym} -
\frac{\inner{\svec{f}}{\svec{\Ym}}}{2\inner{\Yr}{\Yi}}\J\svec{\Yp},
\end{aligned}\end{equation}
and we note that $\calP_c\svec{f}$ satisfies the orthogonality conditions.

\begin{lemma}[Ability to Modulate]\label{Lemma-Modulation}
Fix $\omega_1\in \calI$, $\theta_1\in \R$ and Banach space $L^r$,
$1\le r < \infty$. Then, for $u$ in some neighbourhood of
$\phi_{\omega_1}e^{i\theta_1}$, there exists a Lipschitz map
$u\mapsto (\theta,\omega,\Bp,\Bm)$ such that, under decomposition
\eqref{Eqn-Decomp}, orthogonality conditions
\eqref{Orthog1}-\eqref{Orthog4} are satisfied.
The radius of the neighbourhood is of the order $\dist(\omega_1, \partial\calI)$ as $\dist(\omega_1, \partial\calI)\to 0^+$.
\end{lemma}

For our application we take $r=p=m_2+1$. In particular,
$\abs{\omega_0-\omega(0)}\lesssim \alpha$ and the modulation
parameters are continuous in time. 

\begin{proof}
Define the map $\rho : (\theta,\omega, u)\in \R\times\R \times L^r \to \rho(\theta,\omega, u)\in\R^2$ by
\[\begin{aligned}
\rho (\theta,\omega, u)=
\left(\inner{\svec{\varepsilon}}{\J\twovec{0}{\phi_\omega}},
\inner{\svec{\varepsilon}}{\J\twovec{\partial_\omega\phi_\omega}{0}}
\right) &&\text{ where } &&\varepsilon = ue^{-i \theta} - \phi_\omega.
\end{aligned}\]
At $\left(\theta_1,\omega_1,\phi_{\omega_1} e^{i\theta_1}\right)$, $\rho = 0$, and the Jacobian with respect to $(\theta,\omega)$ is
\[\left[\begin{matrix}
0&-\left.\frac{1}{2}\partial_\omega\norm{\phi}_{L^2}^2\right|_{\omega=\omega_1}\\
\left.\frac{1}{2}\partial_\omega\norm{\phi}_{L^2}^2\right|_{\omega=\omega_1}&0
\end{matrix}\right].\]
The Jacobian is nonsingular and its inverse is uniformly bounded over $\omega_1\in\calI$
by assumption. We may apply the implicit function theorem on
Banach spaces (e.g. Berger~\cite{Berger}) to solve for $(\theta(u),\omega(u))$ such
that $\rho(\theta(u),\omega(u), u)=0$. We then define $\eta$,
$\Bp$, and $\Bm$  by \eqref{Eqn-EtaProjection} and
\eqref{Eqn-DiscreteProj}.
\end{proof}


\subsection{Critical Time $\TC$ \& Proof Strategy} \label{SubSec-Bootstrap}
For all data under consideration, there exists $\TD>0$ such that we may decompose the solution in the manner of \eqref{Eqn-Decomp} for $t\in[0,\TD)$,
\begin{equation}\label{Eqn-Decomp1}
u(t) = \left(\phi_{\omega(t)} + \Bp(t)\Yp + \Bm(t)\Ym + \eta(t)\right)e^{i\theta(t)}.
\end{equation}
Define a new time scale $\TC\leq\TD$ in terms of persistent good control of $\Bp(t)$,
\begin{equation}\label{Eqn-Defn-TC}
\TC = \sup_{T<\TD}\left\{\begin{aligned}
    \forall t\in[0,T], && \abs{\Bp(t)} < \alpha\bka{t}^{-\muBp}
\end{aligned}\right\}.
\end{equation}
Our proof proceeds on two paths, depending on $\TC$:
\begin{enumerate}
\item $\TC = \TD$

This case is considered in Section \ref{Section-Convergence}. We prove that $\TC=+\infty$ and that the solution converges to the soliton family.

\item $\TC < \TD$

This case is considered in Section \ref{Section-Escape}. We prove that the growth of $\Bp(t)$ cannot be controlled, and that exit from the neighbourhood of the soliton family occurs at $\TEx$, with $\TC<\TEx<\TD$.
\end{enumerate}

\subsection{Dynamic Equations}\label{Subsection-DynamicEqns}

Substitute \eqref{Eqn-Decomp1} into \eqref{Eqn-NLS} and use the
expansion \eqref{Eqn-Expansion}:
\[
i(\dot{\theta}-\omega)\left(\phi+\epsilon \right) +
\dot{\omega}\partial_\omega\phi+ \partial_t \epsilon = -i \left( \scrptL \epsilon +
\calN(\epsilon) \right),
\]
where $\epsilon = \Bp\Yp + \Bm\Ym+\eta$, and $\calN$ was defined in \eqref{Eqn-Expansion}. In vector form,
\begin{equation}\label{Eqn-Full}\begin{aligned}
\partial_t&\left( \svec{\eta}+\Bp\svec{\Yp}+\Bm\svec{\Ym} \right)
     + \dot{\omega}\twovec{\partial_\omega\phi}{0}
     + (\dot{\theta}-\omega) \twovec{0}{\phi} \\
     & =\J\boldL\svec{\epsilon} + (\dot{\theta}-\omega)\J\svec{\epsilon} + \J N(\epsilon),
\end{aligned}\end{equation}
where we use the notation $N(\epsilon)= \svec{\calN(\epsilon)}$ and write $\partial_t\svec{\epsilon}$ in full to emphasize its terms will be handled differently.
Take the product of \eqref{Eqn-Full} by $\J\twovec{0}{\phi}$ and use \eqref{Orthog1} to integrate
$\inner{\partial_t \svec{\eta} }{\J\twovec{0}{\phi}}$
by parts in time. We get
\begin{equation}\label{Eqn-OmegaDynamic}\begin{aligned}
\dot{\omega}\inner{\partial_\omega\phi}{\phi} =
    &\dot{\omega}\inner{\svec{\epsilon} }{ \twovec{\partial_\omega\phi}{0} }
    +(\dot{\theta}-\omega)\inner{\svec{\epsilon} }{\twovec{0}{\phi}}
    +\inner{N(\epsilon)}{\twovec{0}{\phi}} .
\end{aligned}\end{equation}

Similarly, with $-\J\twovec{\partial_\omega\phi}{0}$ and using \eqref{Orthog2}, we get
\begin{equation}\label{Eqn-ThetaDynamic}
(\dot{\theta}-\omega)\inner{\partial_\omega\phi}{\phi} =
    \dot{\omega}\inner{\svec{\epsilon} }{ \twovec{0}{\partial^2_\omega\phi} }
    -(\dot{\theta}-\omega)\inner{\svec{\epsilon}}{\twovec{\partial_\omega\phi}{0}}
    -\inner{N(\epsilon)}{\twovec{\partial_\omega\phi}{0}} .
\end{equation}

Now with $\J\svec{\Yp}$ and using \eqref{Orthog3} to remove the terms in $\partial_t\eta$ and $\boldL\eta$, we obtain (recall that $2\inner{\Yr}{\Yi} = 1$):
\begin{equation}\label{Eqn-BmDynamic}\begin{aligned}
\dBm
=
&\EVm\Bm \\
&+\frac{1}{2\inner{\Yr}{\Yi}}\left(
    \begin{aligned}
    &\dot{\omega}\left( \inner{\svec{\eta}}{\partial_\omega\J\svec{\Yp}}
                            - \inner{\Bp\partial_\omega\svec{\Yp} + \Bm\partial_\omega\svec{\Ym}}{\J\svec{\Yp}} \right)\\
    &+\inner{\left(\dot{\theta}-\omega\right)\svec{\epsilon} + N(\epsilon)}{\svec{\Yp}}
    \end{aligned}\right).
\end{aligned}\end{equation}

Finally, with $\J\svec{\Ym}$ and using \eqref{Orthog4} to remove the terms in $\partial_t\eta$ and $\boldL\eta$, we get:
\begin{equation}\label{Eqn-BpDynamic}\begin{aligned}
\dBp
=
&\EVp\Bp \\
&-\frac{1}{2\inner{\Yr}{\Yi}}\left(
    \begin{aligned}
    &\dot{\omega}\left( \inner{\svec{\eta}}{\partial_\omega\J\svec{\Ym}}
                            - \inner{\Bp\partial_\omega\svec{\Yp} + \Bm\partial_\omega\svec{\Ym}}{\J\svec{\Ym}} \right)\\
    &+\inner{\left(\dot{\theta}-\omega\right)\svec{\epsilon} + N(\epsilon)}{\svec{\Ym}}
    \end{aligned}\right).
\end{aligned}\end{equation}

\subsection{Tracking-Error Equation} \label{Subsection-TrackingError}
It will prove more straightforward to estimate the tracking-error $\eta$ in terms of a fixed operator.  Let $\calP_{c,T}$ denote the projection onto the continuous spectrum of the operator with some fixed $\omega(T)$.\footnotemark
\footnotetext{The convergence of $\omega(t)$ as $t\to\TC$ will be established by Lemma \ref{Lemma-thetaomega}.}
Denote $\widetilde{\eta} = \calP_{c,T}\eta$.
We do {\bf not} change our choice of decomposition.

We first isolate linear terms in $\eta$ in \eqref{Eqn-Full}, so that
\[
  \begin{split}
\pd_t \svec{\eta} &= \J \boldL \svec{\eta} + (\dot \theta - \omega) \J
\svec{\eta} + B_0 + \J N  ,\quad \text{where}
\\
B_0 &=
- \dot{\omega}\twovec{\partial_\omega\phi}{0} -
(\dot{\theta}-\omega) \twovec{0}{\phi}
+ \left\{ \J \boldL -\pd_t  + (\dot \theta - \omega)\J \right\}
(\Bp \svec{Y_+} + \Bm\svec{Y_-}).
  \end{split}
\]
Note that $\boldV \equiv \boldL - (-\Delta +V_0 + \omega)\boldI$ is a localized potential.
We further rewrite
\[
  \begin{split}
\pd_t \svec{\eta}
=  \J \boldL_T \svec{\eta} + (\dot \theta - \omega_T ) \J
\svec{\eta} + \J (\boldV - \boldV_T)\svec{\eta}
+ B_0 + \J N.
  \end{split}
\]
Applying $\calP_{c,T}$ to all terms, we get
\begin{equation}\label{Eqn-EtaTilde}\begin{aligned}
\partial_t\svec{\widetilde{\eta}}
    = & \J\boldL_T\svec{\widetilde{\eta}} +\left(\dot{\theta}-\omega_T\right)\calP_{c,T}\J\svec{\widetilde{\eta}}
\\
&+\calP_{c,T}\left(
  -\left(\dot{\theta}-\omega_T\right)\J\left(\calP_{c,T}-\calP_c\right)\svec{\eta}
  + \J (\boldV - \boldV_T)\svec{\eta}
  +  B
  + \J N
  \right),
\end{aligned}
\end{equation}
where $\calP_{c,T}B = \calP_{c,T}B_0$ and
\[
\begin{aligned}
B
=
&
-(\calP_{c,T}-\calP_c)\left(
    \dot{\omega}\twovec{\partial_\omega\phi}{0}
+ \left(\dot{\theta}-\omega\right) \twovec{0}{\phi}
\right)
\\
&
-(\calP_{c,T}-\calP_c)\left(
        (\dBp-\EVp\Bp)\svec{\Yp} + (\dBm-\EVm\Bm)\svec{\Ym}
    \right)
\\
&-\dot{\omega}\left(\Bp\partial_\omega\svec{\Yp}
+ \Bm\partial_\omega\svec{\Ym} \right)
+ \left(\dot{\theta}-\omega\right) \J
\left(\Bp\svec{\Yp}+\Bm\svec{\Ym}\right).
\end{aligned}
\]
Observe from \eqref{Eqn-DiscreteProj} that both $\boldV-\boldV_T$ and $\calP_{c,T}-\calP_c$ are localized potentials, of the order $\abs{\omega-\omega_T}$, depending on $\calI$ and provided $\abs{\omega-\omega_T}$ is sufficiently small. Also note that $B$ has localized spatial support, and
\begin{equation}\label{Eqn-BEstimate}
B = \bigO\left( (\abs{\omega-\omega_T}+\abs{\Bp}+\abs{\Bm}) (\abs{\dot{\omega}} + \abs{\dot{\theta}-\omega} + \abs{\dBp-\EVp\Bp} + \abs{\dBm-\EVm\Bm} ) \right).
\end{equation}

\subsection{Buslaev-Perelman Estimate} \label{Subsec-BPEstimate}
Later we will assume the tracking-error $\widetilde{\eta}$ is small, and then the evolution given by \eqref{Eqn-EtaTilde} should be essentially linear.  To handle the leading order correction in $(\dot{\theta}-\omega)$, we use a technique introduced by Buslaev and Perelman~\cite{Buslaev1995}, later discussed by Buslaev and Sulem~\cite{Buslaev2003}. For $N\geq 3$, the proof is given by Cuccagna~\cite{Cuccagna2003}. For $N=2$, the proof is claimed by Cuccagna and Tarulli~\cite{CuccagnaTarulli2009}. 
Let $\calP_+$ and $\calP_-$ denote the spectral projection operators onto the positive and negative continuous spectrum. That is, $\calP_{c,T} = \calP_+ + \calP_-$.
\begin{proposition}\label{Prop-BuslaevPerelman}
\begin{equation}\label{Eqn-BuslaevPerelman}
-\calP_{c,T}\J = \calP_{c,T}\J^{-1} = i\calP_+ - i\calP_- + \K,
\end{equation}
where $\K$ is a localizing operator, bounded $L^q \mapsto L^{r'}$, for any pair $1< r'\leq 2\leq q<\infty$.
\end{proposition}
\begin{remark}\label{Remark-BuslaevPerelman}
To motivate \eqref{Eqn-BuslaevPerelman}, note that an equation of the form $\partial_t\svec{f} = \J\boldL_T\svec{f}$ may also be written in the form
\[
i\partial_t\twovec{f}{\overline{f}} = \left[\begin{matrix}-\laplacian+\omega-V_1&V_2\\V_3&\laplacian-\omega+V_4\end{matrix}\right]\twovec{f}{\overline{f}},
\]
where the potentials $V_j$ are due to both $V_0$ and the nonlinearity.  We may view this matrix operator as a perturbation of $\dis \left[\begin{matrix}-\laplacian+\omega+V_0&0\\0&\laplacian-\omega-V_0\end{matrix}\right]$,
for which \eqref{Eqn-BuslaevPerelman} is an identity when $\K=0$.
\end{remark}
Let us decompose $\svec{\widetilde{\eta}} = \calP_{c,T} \svec{\eta}$ according to $\calP_\pm$, and incorporate the accumulated error in tracking the phase,
\[
\widetilde{\eta}_{\pm} = \left(\calP_\pm\svec{\widetilde{\eta}}\right)e^{\mp i\left(\theta(t)-t\omega_T\right)} \equiv \left(\calP_\pm\svec{\widetilde{\eta}}\right)e^{\pm i\Theta(t)}.
\]
The evolution of $\widetilde{\eta}_{\pm}$ follows from \eqref{Eqn-EtaTilde},
\begin{equation}\label{Eqn-EtaPlusMinus}
\partial_t\widetilde{\eta}_\pm = \J\boldL_T\widetilde{\eta}_\pm + \calP_\pm e^{\pm i\Theta(t)}
\left(
  \left(\dot{\theta}-\omega_T\right)\K\svec{\widetilde{\eta}}
  + \J(\boldV-\boldV_T)\svec{\widetilde{\eta}}
  + B
  + \J N\left(\epsilon\right)
\right),
\end{equation}
where we have abused notation to absorb $(\dot{\theta}-\omega_T)\left(\calP_{c,T}-\calP_c\right)$ into $\boldV-\boldV_T$.
Equation \eqref{Eqn-EtaPlusMinus} will be our means to establish $L^p$ and $\Lloc$ estimates for $\widetilde{\eta}$ through the equivalent Duhamel formulation:
\begin{multline} \label{Eqn-EtaPMint}
\widetilde{\eta}_\pm = e^{t\J\boldL_T}\widetilde{\eta}_\pm(0)\\
+\int_0^t e^{(t-s)\J\boldL_T} \calP_\pm e^{\pm i\Theta(s)}\left(
  \left(\dot{\theta}-\omega\right)\K\svec{\widetilde{\eta}}
  + \J(\boldV-\boldV_T)\svec{\widetilde{\eta}}
  + B
  + \J N\left(\epsilon\right)
\right)(s)\,ds.
\end{multline}

\subsection{Arithmetic}
\label{Section-Numerology}
For later reference, we collect here an assortment of arithmetic facts.
Due to  \eqref{Eqn-NonlinCondition1} and \eqref{Eqn-NonlinCondition2}, there exists $\deltaone>0$ sufficiently small such that we may define:
\begin{equation}\begin{aligned}\label{Eqn-SigmaQ}
\muQ &= \left\{\begin{aligned}
                 &\sigma_p &&\text{ for }\sigma_p > \frac{2}{\mZero+1},\\
                 &\frac{2}{\mZero+1}+\deltaone &&\text{ for }\sigma_p \leq \frac{2}{\mZero+1},
        \end{aligned}\right.\\
& \text{ with } \muQ < \min\{1,  \frac{N}{2}(m_1-1)-1 \}.
\end{aligned}\end{equation}
Note that $\muQ\in\left[\sigma_p,1\right)$, corresponding to some $q\in[p,p_{\max})$. Such a choice is only possible for $N\geq 2$, since
\begin{equation}\label{Arith8}
(\mZero+1)\muQ > 2,
\end{equation}
and, in particular, $\sigma_q > \frac{1}{2}$.  Constant $\deltaone>0$ is included to ensure strict inequality for \eqref{Arith8}, and is otherwise neglected by taking it sufficiently small.
Let us emphasize that
\begin{equation}\label{Arith2}
m_1 > 1 + \frac{2}{N}\left(1+\muQ\right).
\end{equation}
Consider $m_j\theta_j\sigma_p \equiv \frac{N}{2}\left(m_j-1\right)-\sigma_p$, an expression that will arise during norm interpolation. Then
\begin{equation}\label{Arith3}\begin{aligned}
m_j\theta_j\sigma_p >1
&&
\mbox{and}
&&
m_j\theta_j > 1,
\end{aligned}\end{equation}
which are due to \eqref{Eqn-NonlinCondition1} and $\sigma_p<1$, respectively. Due to $q<p_{\max}$,
\begin{equation}\label{Arith5}
2\left(1-\frac{1}{q}\right) < 1+\frac{2}{N} < m_j.
\end{equation}
Consider $m_j\widetilde{\theta}_j\sigma_p \equiv \frac{N}{2}\left( m_j - 1\right) - \sigma_q$, another expression that will arise during norm interpolation.
Due to \eqref{Arith2}, and since $\sigma_p<1<m_0$,
\begin{equation}\label{Arith7new}\begin{aligned}
m_j\widetilde{\theta}_j\sigma_p > 1 &&\mbox{and}&& m_j\widetilde{\theta}_j > \frac{1}{m_0}. 
\end{aligned}\end{equation}
Finally, one can verify that $\frac{1}{m_0}\left(1-\frac{m_0}{2}\right)>1-m_j\left(\frac{m_0}{2}\right)$ is a consequence of $m_j\geq m_1>1$, and since $m_1^3-3m_1+2>0$ when $m_0=m_1<2$.
The second expression of \eqref{Arith7new} implies:
\begin{equation}\label{Arith6}
(1-\widetilde{\theta}_j)m_j\left(\frac{\mZero}{2}\right) + m_j\widetilde{\theta}_j > 1.
\end{equation}

\section{Convergence Case}
\label{Section-Convergence}
In this section, we will consider the case of $\TC=\TD$.  We will prove that $\TC=+\infty$ and convergence to a soliton by means of a bootstrap argument. The estimates shown below will be reused in Section \ref{Section-Escape}.

\subsection{Hypotheses and First Estimates}

For $t\in [0,\TD)$, we decompose $u(t)$ as in \eqref{Eqn-Decomp1}. Assume that $\THy\in(0,\TC]$ is the last time for which the following hypotheses hold for all $t\in [0,\THy)$:
\begin{equation} \label{Eqn-Boothypo}
\begin{aligned}
&|\Bm(t)| &&&\leq&& &2C_1\alpha \bka{t}^{-\muBm},\\
&\norm{\eta(t)}_{L^p} &&&\leq&& &2C_2\alpha \bka{t}^{-\muP},\\
&\norm{\eta(t)}_{\Lloc} &&&\leq&& &2C_3\alpha \bka{t}^{-\muLoc}.
\end{aligned}
\end{equation}
Universal constants $C_j$ will be determined later. Recall that from \eqref{Eqn-Defn-TC}, the definition of $\TC$, we also have
\begin{equation} \label{Eqn-BpControl}
|\Bp(t)|\leq \alpha\bka{t}^{-\muBp}.
\end{equation}

Under these hypotheses, the goal is now to obtain the same inequalities as \eqref{Eqn-Boothypo} without the factor $2$, by choosing $\overline{\alpha_0}$ small enough. We now fix $T=\THy$.
First, we prove estimates on the parameters $\theta$ and $\omega$.

\begin{lemma} \label{Lemma-thetaomega}
For all $t\in [0,\THy)$, we have $\omega(t)\in\calI$,
\begin{equation}\label{Eqn-thetaomega} \begin{aligned}
|\dot{\omega}(t)| + |\dot{\theta}(t)-\omega(t)| &\leq C\alpha^{\mZero}\bka{t}^{-\muNon}, &&\text{ which implies }\\
|\omega(t)-\omega(T)| &\leq C\alpha^{\mZero}\bka{t}^{-\muNon+1}, &&\text{ and }\\
|e_{\pm,\omega(t)}-e_{\pm,\omega(T)}| &\leq C\alpha^{\mZero}\bka{t}^{-\muNon+1}.
\end{aligned}\end{equation}
By exactly the same proof, we have
\[
\abs{\dot{b}_\pm - e_\pm b_\pm} \leq C\alpha^{\mZero}\bka{t}^{-\muNon}.
\]
\end{lemma}

\begin{proof}
We will only prove the first inequality. To do this, we add \eqref{Eqn-OmegaDynamic} and \eqref{Eqn-ThetaDynamic}, to obtain $|\dot{\omega}|+|\dot{\theta}-\omega|\leq (|\dot{\omega}|+|\dot{\theta}-\omega|)\O(\epsilon) +\abs{\inner{N(\epsilon)}{\mbox{potential}}}$. With $\overline{\alpha_0}$ small enough, we get $|\dot{\omega}|+|\dot{\theta}-\omega|\leq \inner{N(\epsilon)}{\mbox{potential}}$, and \eqref{Eqn-thetaomega} follows from \eqref{Eqn-NonlinConseq} and the bootstrap inequalities \eqref{Eqn-Boothypo} and \eqref{Eqn-BpControl}.
We use the $\Lloc$ control of $\eta$ since it may be stronger by \eqref{Eqn-SigmaQ}.
\end{proof}

\subsection{Improved Estimates}
\label{Subsection-ImprovedEstimates}

\emph{Estimate on $b_-$.} We can now improve the estimate on $b_-$. Fix $e_1=e_{-,\omega(T)}<0$. From Lemma \ref{Lemma-thetaomega} we obtain:
\[
\abs{\dot{b_-}-e_1b_-} \lesssim \alpha^{\mZero}\bka{t}^{-\muNon}\left( \bka{t}\abs{b_-(t)} + 1 \right) \leq C_0 \alpha^{\mZero}\bka{t}^{-\muNon},
\]
where $C_0>0$ is an explicit constant depending on $\calI$. Moreover, since $\norm{\epsilon(0)}_{H^1\cap L^1} =\alpha$, we may assume $|b_-(0)|\leq C_0\alpha$.
Integrating between $0$ and $t$,
\begin{align*}
|b_-(t)| &\leq e^{+e_1t}C_0\alpha +\int_0^t e^ {+e_1(t-s)}C_0\alpha^{\mZero}\bka{s}^{-\muNon}\,ds\\
         &\lesssim \alpha e^{+e_1t} +\alpha^{\mZero}\left(
           \int_{0}^{t - \frac{\muNon}{\abs{e_1}}\ln t}{e^{+e_-(t-s)}\,ds} + \int_{t - \frac{\muNon}{\abs{e_1}}\ln t}^t{\bka{s}^{-\muNon}\,ds}
           \right),
\end{align*}
with the obvious corrections to integral bounds when $t\approx 0$.
For $\overline{\alpha_0}$ sufficient small, we have shown $\abs{b_-(t)} \lesssim \alpha\left(1+\ln{\bka{t}}\right)\bka{t}^{-\muNon} \leq C_1\alpha\bka{t}^{-\muBm}$
where $C_1>C_0$ is the appropriate universal constant.

\emph{Decomposition of $\eta$.}
Recall from \eqref{Eqn-EtaProjection} that $\eta = \calP_c\eta$, implying
\[
\eta = \widetilde{\eta} + \left(\calP_c - \calP_{c,T}\right)\eta,
\]
so that norms of $\eta$, $\widetilde{\eta}$, and $\widetilde{\eta}_\pm$ are all comparable. Due to \eqref{Orthog1}, we have \[\norm{u(t)}_{L^2}^2 \approx \inner{\phi_\omega+\eta}{\phi_\omega+\eta} = \norm{\phi_\omega}_{L^2}^2+\norm{\eta}_{L^2}^2, \] where we have ignored terms of order $b_\pm(t)$.  From Lemma \ref{Lemma-thetaomega}, we conclude
\begin{equation}\label{Eqn-EtaL2}
\norm{\eta(t)}_{L^2} \leq C_4 \alpha^{\frac{\mZero}{2}},
\end{equation}
where the constant $C_4$ depends on $\calI$.
Now we estimate terms of
\eqref{Eqn-EtaPMint}, in some norm $L^{r'}$, with $r'\in (1,2)$.
From Proposition \ref{Prop-BuslaevPerelman}, \eqref{Arith8}, and Lemma \ref{Lemma-thetaomega}:
\begin{equation} \label{Eqn-Klr}
\norm{\left(\dot{\theta}-\omega_T\right)\K\svec{\widetilde{\eta}}}_{L^{r'}} \leq \left(\dot{\theta}-\omega_T\right) \norm{\widetilde{\eta}}_{L^q} \leq C\alpha^{2+}\bka{t}^{-(m_0+1)\muQ+1} \lesssim \alpha^{2+}\bka{t}^{-(1+)}.
\end{equation} 
Note that we use here standard notation $1+$ to designate a number slightly bigger than $1$.
The same estimate holds for both $\norm{\J(\boldV-\boldV_T)\svec{\widetilde{\eta}}}_{L^{r'}}$ and $\norm{B}_{L^{r'}}$. Indeed, from \eqref{Eqn-BEstimate}, we have
\begin{equation}\label{Eqn-Blr}
\norm{B}_{L^{r'}} \leq C \alpha^{2+}\left( \bka{t}^{-2\muNon+1} + \bka{t}^{-\muNon-\muBp}\right),
\end{equation}
which is lower order provided $\sigma_q<1$.

\emph{$L^p$ estimate of $\eta$.}
To start, we estimate $\norm{N(\epsilon)}_{L^{p'}}$ using equation \eqref{Eqn-NonlinConseq}:
\[
\norm{\epsilon^{m_j}}_{L^{p'}} \lesssim \left(\abs{b_+}+\abs{b_-}\right)^{m_j} + \norm{\eta}_{L^{m_jp'}}^{m_j},
\]
which may be interpolated as
\begin{equation}\label{Eqn-LpInterp}
\norm{\eta}_{L^{m_jp'}}^{m_j} \leq \norm{\eta}_{L^2}^{(1-\theta_j)m_j} \norm{\eta}_{L^p}^{m_j\theta_j},
\end{equation}
with $\dis \frac{1}{m_jp'} = \frac{\theta_j}{p}+\frac{1-\theta_j}{2}$, or, equivalently, $\dis \theta_j = \frac{\frac{1}{2}-\frac{1}{m_jp'}}{\frac{1}{2}-\frac{1}{p}}$.
We note that $\theta_j \in (0,1]$ provided $\dis 2\left(1-\frac{1}{p}\right)<m_j\leq p-1$, which is implied by $1+\frac{2}{N}<m_j\leq m_2$.
With \eqref{Eqn-Boothypo}, \eqref{Eqn-BpControl} and \eqref{Arith3}, we have
\begin{equation}\label{Eqn-Epslr}
\norm{\epsilon^{m_j}}_{L^{p'}} \lesssim \alpha^{1+}\bka{t}^{-(1+)}.
\end{equation}
If $m_j > 2$,
\[
\norm{A_j\phi^{m_j-2}\epsilon^2}_{L^{p'}} \lesssim \left(\abs{b_+}+\abs{b_-}\right)^2 + \norm{\eta}_{\Lloc}^2 \leq \alpha^2\bka{t}^{-(1+)},
\]
which we view as a correction to \eqref{Eqn-Epslr}.

Apply \eqref{Eqn-AssumedDecay} to all terms of \eqref{Eqn-EtaPMint}, with $r=p$. Since $L^p$ is energy subcritical, we may use that $\norm{\etat_\pm(0)}_{H^1}$ is small to improve the bound on the linear term,
\[
\norm{e^{t\J\boldL_T}\etat_\pm(0)}_{L^p} \leq \frac{C_2}{4} \alpha\bka{t}^{-\sigma_p}.
\]
The universal constant $C_2$ is determined by this relation. For the other terms, apply \eqref{Eqn-Klr}, \eqref{Eqn-Blr} and \eqref{Eqn-Epslr} to get
\begin{equation}\label{Eqn-LpIntegral}\begin{aligned}
\norm{\etat_{\pm}(t)}_{L^p}
&\leq \frac{C_2}{4}\alpha \bka{t}^{-\sigma_p} + C\alpha^{1+} \int_0^t{ |t-s|^{-\sigma_p}\bka{s}^{-(1+)}\, ds}.
\end{aligned}\end{equation}
This proves $\norm{\eta(t)}_{L^p}\leq C_2\alpha \bka{t}^{-\sigma_p}$, as desired, by assuming $\overline{\alpha_0}>0$ is sufficiently small.

\emph{$\Lloc$ estimate of $\eta$.} We consider now the $\Lloc$ norm of $\eta$. 
From \eqref{Eqn-EtaPMint}, we have
\begin{equation}\label{Eqn-L2Split}\begin{aligned}
\norm{\widetilde{\eta}_{\pm}(t)}_{\Lloc}
&\lesssim \norm{e^{t\J\boldL_T}\widetilde{\eta}_{\pm}(0)}_{L^q} \\
&\quad + \int_0^{t} \norm{e^{(t-s)\J\boldL_T}\calP_{\pm}\left( \left(\dot{\theta}-\omega\right)\K\svec{\widetilde{\eta}} +\J(\boldV-\boldV_T)\svec{\widetilde{\eta}} + B \right)}_{L^{q}}ds\\
&\quad + \int_0^{t} \norm{e^{(t-s)\J\boldL_T} \calP_{\pm}N(\epsilon)}_{L^q}\,ds \equiv \mathbf{I}+\mathbf{II}+\mathbf{III}.
\end{aligned}\end{equation}
To estimate these terms, we use dispersive estimate \eqref{Eqn-AssumedDecay}. We first get
\begin{equation}\label{Eqn-L2-termI}
\mathbf{I} \lesssim \bka{t}^{-\sigma_q}\norm{\widetilde{\eta}_{\pm}(0)}_{L^1\cap H^1} \leq \frac{C_3}{2}\alpha \bka{t}^{-\muLoc}.
\end{equation} 
The universal constant $C_3$ is determined by this relation.

Term $\mathbf{III}$ is treated in a similar way as for the $L^p$ estimate. We have
\[
\norm{\epsilon^{m_j}}_{L^{q'}} \lesssim \left(\abs{b_+}+\abs{b_-}\right)^{m_j} + \norm{\eta}_{L^{m_jq'}}^{m_j},
\]
where now we interpolate according to
\begin{equation}\label{Eqn-L2Interp}
\norm{\eta}_{L^{m_jq'}}^{m_j} \leq \norm{\eta}_{L^2}^{(1-\widetilde{\theta}_j)m_j} \norm{\eta}_{L^p}^{m_j\widetilde{\theta}_j},
\end{equation}
with $\dis \frac{1}{m_jq'} = \frac{1-\widetilde{\theta}_j}{2}+\frac{\widetilde{\theta}_j}{p}$, or, equivalently, $\dis \widetilde{\theta}_j = \frac{\frac{1}{2}-\frac{1}{m_jq'}}{\frac{1}{2}-\frac{1}{p}}$.
We note that $\widetilde{\theta}_j\in(0,1]$ provided $\dis 2\left(1-\frac{1}{q}\right) < m_j \leq p\left(1-\frac{1}{q}\right)$.
The former inequality is \eqref{Arith5}, while the latter inequality is true for any $q\geq p\geq m_j+1$.
From \eqref{Eqn-EtaL2}, the terms of \eqref{Eqn-L2Interp} give $\alpha$ with the exponent $(1-\widetilde{\theta}_j)m_j\frac{\mZero}{2}+m_j\widetilde{\theta}_j$, which is greater than $1$ by \eqref{Arith6}.  We have
\begin{equation}\label{Eqn-Epslr2}
\norm{\epsilon^{m_j}}_{L^{q'}} \lesssim \alpha^{1+}\bka{t}^{-m_j\widetilde{\theta}_j\sigma_p},
\end{equation}
and if $m_j> 2$,
\begin{equation}\label{Eqn-Epslr2Part2}
\norm{A_j\phi^{m_j-2}\epsilon^2}_{L^{q'}} \lesssim \left(\abs{b_+}+\abs{b_-}\right)^2 + \norm{\eta}_{\Lloc}^2 \leq \alpha^2\bka{t}^{-2\muLoc},
\end{equation}
and both decay faster than $\bka{t}^{-1}$, due to \eqref{Arith7new} and \eqref{Arith8}, respectively.  With \eqref{Eqn-AssumedDecay},
\begin{equation}\label{Eqn-L2-termIII}\begin{aligned}
\mathbf{III} &\lesssim  \alpha^{1+} \int_0^{t} \abs{t-s}^{-\sigma_q} \bka{s}^{-(1+)}\,ds
    \lesssim \alpha^{1+} \bka{t}^{-\muLoc}.
\end{aligned}\end{equation}
The integral is a simple calculation since $\sigma_q<1$.
Term {\bf II} may be included with {\bf III}.  Its contribution is controlled by \eqref{Eqn-Klr} and \eqref{Eqn-Blr} applied with $r=q$.

Assuming $\overline{\alpha_0}>0$ is sufficiently small, we have shown that $\norm{\eta(t)}_{\Lloc}\leq C_3\alpha \bka{t}^{-\muLoc}$.

\subsection{Bootstrap Conclusions}

It was shown in the previous section that, for all $t\in [0,\THy)$,
\begin{equation} \label{Eqn-Bootconcl}
\begin{aligned}
&|\Bm(t)| &&&\leq&& &C_1\alpha \bka{t}^{-\muBm},\\
&\norm{\eta(t)}_{L^p} &&&\leq&& &C_2\alpha \bka{t}^{-\muP},\\
&\norm{\eta(t)}_{\Lloc} &&&\leq&& &C_3\alpha \bka{t}^{-\muLoc}.
\end{aligned}
\end{equation}
As these are continuously evolving quantities, by \eqref{Eqn-Boothypo} it must be that $\THy = \TC$. Together with the estimate for $b_+$, this proves proximity to the soliton family as $t\to\TC$, and so by Lemma \ref{Lemma-Modulation} the decomposition can be extended past $t=\TC$. Therefore, if $\TC=\TD$, then $\TD=+\infty$.

By integration of the first equation of \eqref{Eqn-thetaomega}, there exists $\omega_+\in\calI$, $\abs{\omega(0)-\omega_+}\lesssim \alpha^{\mZero}$ such that the remaining estimates of \eqref{Eqn-thetaomega} hold.  It is essential here that $\muNon > 1$.
From~\eqref{Eqn-Bootconcl}, we have that $\norm{u(t)-\phi_{\omega(t)}e^{i\theta(t)}}_{L^p} \leq C\alpha\bka{t}^{-\muP}$. Together we have, for all $t>0$,
\[
\norm{u(t)-\phi_{\omega_+}e^{i\theta(t)}}_{L^p} \lesssim \alpha\bka{t}^{-\muP} + \alpha^{\mZero}\bka{t}^{-\muNon+1}.
\]
For Theorem \ref{Thm-MainResult}, we use $\muThm = \min\{\muP,\muNon-1\}>0$.

\section{Escape Case}
\label{Section-Escape}
In this section, we consider the case of $\TC<\TD$. The arguments of Section \ref{Section-Convergence} apply for $t\in[0,\TC)$. We extend these arguments to prove that the parameter $\abs{b_{+,\omega(t)}}$ grows exponentially for an interval of time after $\TC$.

\subsection{Hypotheses and First Estimates}
Assume that $\TEx \in(\TC,\TD]$ is the last time for which the following hypotheses \eqref{Eqn-Boothypo2}-\eqref{Eqn-BoothypoF} hold for all $t\in[0,\TEx)$:
\begin{equation} \label{Eqn-Boothypo2}
\begin{aligned}
&|\Bm(t)| &&&\leq&& &2C_5\abs{\Bp(t)},\\  
&\norm{\eta(t)}_{L^p} &&&\leq&& &2C_6 \left( \abs{\Bp(t)}^{\mZero} + \alpha \bka{t}^{-\muP}\right),\\
&\norm{\eta(t)}_{\Lloc} &&&\leq&& &2C_7\left( \abs{\Bp(t)} + \alpha \bka{t}^{-\muLoc} \right).
\end{aligned}
\end{equation}
Universal constants $C_j$ will be determined later. In place of \eqref{Eqn-Defn-TC}, we make two further hypotheses: for all $s<t\in[\TC,\TEx)$,
\begin{equation} \label{Eqn-Boothypo3}
\begin{aligned}
e^{\frac{2}{3}e_2(t-s)} &&\leq&& &\frac{\Bp(t)}{\Bp(s)} \leq e^{\frac{3}{2}e_2(t-s)},
\end{aligned}
\end{equation}
where $e_2 = e_{+,\omega(T)}$, the positive eigenvalue of $\boldL_T$ for $T=\TC$. Recall that $e_1 = e_{-,\omega(T)} = -e_2 <0$. Our final hypothesis is, for all $t\in[0,\TEx)$,
\begin{equation} \label{Eqn-BoothypoF}
\begin{aligned}
\abs{\Bp(t)}<3\alpha_0.
\end{aligned}
\end{equation}
As a particular consequence of \eqref{Eqn-Boothypo3} and \eqref{Eqn-Defn-TC}, note that for $t\in(\TC,\TEx)$,
\begin{equation}\label{Eqn-BpTrel}
\abs{\Bp(t)} > \abs{\Bp(\TC)} \geq \alpha\bka{\TC}^{-1} > \alpha\bka{t}^{-1}.
\end{equation}

Under these hypotheses, our goal is to obtain hypotheses \eqref{Eqn-Boothypo2} without the factor 2, and hypothesis \eqref{Eqn-Boothypo3} with tighter exponents.  We will argue that hypothesis \eqref{Eqn-BoothypoF} {\bf cannot} be improved, and conclude the expected exit behavior.

\begin{lemma}\label{Lemma-thetaomega2}
For all $s<t\in[\TC,\TEx)$, $\omega(t)\in\calI$ and:
\[ \begin{aligned}
|\dot{\omega}(t)| + |\dot{\theta}(t)-\omega(t)|&\leq C\left( \abs{b_+(t)}^{\mZero} + \alpha^{\mZero}\bka{t}^{-\muNon}\right), &&\text{ and }\\
|\omega(t)-\omega(s)| &\leq C\left( \abs{b_+(t)}^{\mZero} + \alpha^{\mZero} \bka{s}^{-\muNon+1} \right), &&\text{ which implies } \\
|e_{\pm,\omega(t)}-e_{\pm,\omega(s)}| &\leq C\left( \abs{b_+(t)}^{\mZero} + \alpha^{\mZero} \bka{s}^{-\muNon+1} \right).
\end{aligned}\]
\end{lemma}
\begin{proof}The argument is the same as for Lemma \ref{Lemma-thetaomega}, based on the following estimate:
$\dis
\abs{\inner{N(\epsilon)}{\textrm{potential}}} \lesssim \abs{\Bp(t)}^{\mZero} + \alpha^{\mZero}\bka{t}^{-\muNon}
$.
When integrating quantities in $\abs{\Bp(t)}^r$, for any exponent $r>0$, we use \eqref{Eqn-Boothypo3}:
\begin{equation}\label{Eqn-BpInt}
\int_s^t{\abs{\Bp(\tau)}^r\,d\tau} \leq \abs{\Bp(t)}^r\int_s^te^{-\frac{2}{3}r\,e_2(t-\tau)}\,d\tau. \qedhere
\end{equation}
\end{proof}

\subsection{Improved Estimates}
\label{Subsection-ImprovedEstimates2}

\emph{Growth estimate for $b_+$.}
By Lemma \ref{Lemma-thetaomega2}, the dominant forcing term of \eqref{Eqn-BpDynamic} is $\inner{N(\epsilon)}{\svec{\Ym}}$, and it reads: $\dot{b}_+ - e_2b_+ = \bigO\left(\abs{b_+}^{\mZero} + \alpha^{\mZero}\bka{t}^{-\muNon}\right)$.
Taking $\overline{\alpha_0}$ to be sufficiently small, we may assume $|\dot{b}_+ - e_2b_+|<\frac{1}{5}e_2\abs{b_+}$. After integration, this is a stronger statement than~\eqref{Eqn-Boothypo3}:
\begin{equation}\label{Eqn-Bootconc3}
e^{\frac{4}{5}e_2(t-s)} \leq \frac{\Bp(t)}{\Bp(s)} \leq e^{\frac{6}{5}e_2(t-s)}.
\end{equation}

\emph{Estimate on $b_-$.} The same argument applied to \eqref{Eqn-BmDynamic} gives
\[
\abs{b_-(t)} \leq e^{+e_1(t-\TC)}\abs{b_-(\TC)} + \int_{\TC}^te^{+e_1(t-\tau)}\abs{b_+(\tau)}\,d\tau.
\]
Compare the estimate for $\Bm$ from the previous chapter \eqref{Eqn-Boothypo} with \eqref{Eqn-Defn-TC} to see that $\abs{b_-(\TC)} \leq C_1 \abs{b_+(\TC)}$.
We conclude that $\abs{b_-(t)} \leq C_5 \abs{b_+(t)}$ for some $C_5$.

\emph{Decomposition of $\eta$.}
Analogous to \eqref{Eqn-EtaL2}, and using \eqref{Orthog0}-\eqref{Orthog4}, we have
\begin{equation}\label{Eqn-EtaL2-redux}
\norm{\eta(t)}_{L^2}\leq C_8\left(\abs{b_+(t)}^\frac{\mZero}{2} + \alpha^\frac{\mZero}{2}\right).
\end{equation}
Under our new hypotheses, let us revisit \eqref{Eqn-Klr} and \eqref{Eqn-Blr}.  As before, we use \eqref{Eqn-BEstimate}, Lemma~\ref{Lemma-thetaomega2} and \eqref{Arith8} to conclude:
\begin{equation}\label{Eqn-Klr-Blr-redux}
\norm{\left(\dot{\theta}-\omega_T\right)\K\svec{\widetilde{\eta}}}_{L^{r'}} +\norm{\J(\boldV-\boldV_T)\svec{\widetilde{\eta}}}_{L^{r'}} + \norm{B}_{L^{r'}} \ll \left(\abs{b_+(t)}^{\mZero}+\alpha^{\mZero}\bka{t}^{-(1+)}\right).
\end{equation}

\emph{$L^p$ estimate of $\eta$.}
For $t\in[\TC,\TEx)$, we estimate $\norm{N(\epsilon)}_{L^{p'}}$  using \eqref{Eqn-LpInterp} and \eqref{Eqn-Boothypo2}:
\begin{equation}\label{Eqn-LpInterp-redux}
\norm{N(\epsilon)}_{L^{p'}} \lesssim \abs{b_+}^{\mZero} + \sum_j\norm{\eta}_{L^2}^{(1-\theta_j)m_j} C_6^{m_j\theta_j}\left(\abs{b_+(t)}^{\mZero}+\alpha\bka{t}^{-\muP}\right)^{m_j\theta_j}.
\end{equation}
Due to \eqref{Arith3} and \eqref{Eqn-BoothypoF}, we may bound \eqref{Eqn-LpInterp-redux} by $\frac{C_6}{C_9}\left( \abs{b_+}^{\mZero} + \alpha^{1+}\bka{t}^{-(1+)}\right)$, for any universal constant $C_9>0$,
by taking the universal constant $C_6$ sufficiently large and $\overline{\alpha_0}$ sufficiently small.
Now apply \eqref{Eqn-AssumedDecay} and use both \eqref{Eqn-LpInterp-redux} and, for $t\in[0,\TC)$, the estimates from the previous chapter that led to \eqref{Eqn-LpIntegral}:
\[ \norm{\etat_{\pm}(t)}_{L^p} \leq \frac{C_2}{2}\alpha \bka{t}^{-\sigma_p} + C\alpha^{1+} \int_0^t{ \abs{t-s}^{-\sigma_p}\bka{s}^{-(1+)}\, ds} + \frac{C_6}{C_9}\int_{\TC}^t{ \abs{t-s}^{-\sigma_p}\abs{b_+(s)}^{\mZero}\,ds}. \]
From \eqref{Eqn-LpIntegral}, the first terms are bounded by $C_2\alpha\bka{t}^{-\sigma_p}$.  Assume that $C_6\geq C_2$. Integrate the final term with \eqref{Eqn-BpInt}, and take $C_9$ to be the constant factor. This completely determines~$C_6$.  We have shown $\norm{\etat_{\pm}(t)}_{L^p} < C_6\left(\alpha\bka{t}^{-\sigma_p}+\abs{b_+(t)}^{\mZero}\right)$.

\emph{$\Lloc$ estimate of $\eta$.} Recall the terms {\bf I}, {\bf II} and {\bf III} of \eqref{Eqn-L2Split}. The estimate \eqref{Eqn-L2-termI} of term {\bf I} still applies.
For term {\bf III}, we first consider \eqref{Eqn-L2Interp} for $t>\TC$, using \eqref{Eqn-Boothypo2}, \eqref{Eqn-BpTrel} and \eqref{Eqn-EtaL2-redux},
\[ \begin{aligned}
\norm{\eta}_{L^{m_jq'}}^{m_j}
&\lesssim o(1)\left(\abs{b_+}^{\mZero}+\alpha\bka{t}^{-\muP}\right)^{m_j\widetilde{\theta_j}} \lesssim \abs{b_+(t)}^{1+},
\end{aligned}\]
where the second inequality relied on \eqref{Arith7new} and \eqref{Arith6} for the exponents of $\Bp$ and $\alpha$, respectively.
If $m_j > 2$,
\[
\norm{A_j\phi^{m_j-2}\epsilon^2}_{L^{q'}} \lesssim \abs{b_+}^2 + \alpha^2\bka{t}^{-2\muLoc},
\]
as in \eqref{Eqn-Epslr2Part2}. As before, all terms decay faster than $\bka{t}^{-1}$ and we have:
\[ \begin{aligned}
{\bf III}
\lesssim \int_{\TC}^{t}{\abs{b_+(s)}^{1+}\,ds} +  \alpha^{1+}\int_0^{t}{\abs{t-s}^{-\sigma_q}\bka{s}^{-(1+)}\,ds},
\end{aligned}\]
which are the same integrals as \eqref{Eqn-BpInt} and \eqref{Eqn-L2-termIII}, respectively.
Term {\bf II} may be included for $s\in[0,t)$, controlled by a combination of \eqref{Eqn-Klr}, \eqref{Eqn-Blr} and \eqref{Eqn-Klr-Blr-redux}:
\[
{\bf II} + {\bf III} \lesssim  \abs{b_+(t)}^{1+} + \alpha^{1+}\bka{t}^{-\sigma_q}.
\]
This concludes our estimate of $\norm{\widetilde{\eta}_\pm}_{\Lloc}$.

\subsection{Bootstrap Conclusions}
In the previous section, we proved \eqref{Eqn-Bootconc3} and that, for all $t\in[\TC,\TEx)$,
\begin{equation}\label{Eqn-Bootconc2}
\begin{aligned}
&|\Bm(t)| &&&\leq&& &C_5\abs{\Bp(t)},\\
&\norm{\eta(t)}_{\Lloc} &&&\leq&& &C_7\left(\abs{\Bp(t)} + \alpha \bka{t}^{-\muLoc}\right),\\
&\norm{\eta(t)}_{L^p} &&&\leq&& &C_6\left(\abs{\Bp(t)}^{\mZero} + \alpha \bka{t}^{-\sigma_p}\right).
\end{aligned}
\end{equation}
These are continuously evolving quantities.  From \eqref{Eqn-BoothypoF}, \eqref{Eqn-BpTrel} and \eqref{Eqn-Bootconc3}, we conclude that
\[
\TEx \leq \TC + \frac{5}{4e_2}\ln\left(3\frac{\alpha_0}{\alpha}\bka{\TC}\right) < +\infty.
\]
Assuming $\overline{\alpha_0}$ is sufficiently small,  Lemma \ref{Lemma-Modulation}, \eqref{Eqn-BoothypoF} and \eqref{Eqn-Bootconc2} imply that the decomposition can be extended past time $\TEx$, and hence $\TEx < \TD$. The only possible failure at time $t=\TEx$ is \eqref{Eqn-BoothypoF}, and so we conclude that $b_+(\TEx) = 3\alpha_0$.  The conclusion of Theorem \ref{Thm-MainResult} then follows for some $\omega_+ \approx \omega(\TEx)$.

\section*{Acknowledgments}
We would like to thank Galina Perelman for Remark \ref{Remark-BuslaevPerelman}, and for bringing the work of Beceanu to our attention.  We also thank Eduard-Wilhelm Kirr for discussion of his papers.

\section*{Funding}
This work was partially supported by the Natural Sciences and
Engineering Research Council of Canada [261356-08 to T.-P.~T.]; and
the Pacific Institute for the Mathematical Sciences [through
  fellowships to V.~C. and I.~Z.].

\end{document}